\documentclass{article}

\usepackage{color}
\usepackage{amsthm}
\usepackage{tikz}
\usepackage{hyperref}

\usetikzlibrary{decorations.pathreplacing}

\usepackage{amsmath}
\usepackage{amsfonts}
\usepackage{amstext}
\usepackage[latin1]{inputenc}
\usepackage{amscd}
\usepackage{latexsym}
\usepackage{bm}
\usepackage{amssymb}
\usepackage[all]{xy}
\usepackage{euscript}
\usepackage{a4wide}
\usepackage{amsmath,amssymb,graphicx}
\usepackage{amssymb}
\usepackage{amsmath}
\usepackage{mathrsfs,mathtools}

\newtheorem{theorem}{Theorem}
\newtheorem{proposition}{Proposition}

\newtheorem{lemma}{Lemma}
\newtheorem{corollary}{Corollary}
\newtheorem{example}{Example}
\newtheorem{definition}{Definition}
\newtheorem{remark}{Remark}

\def\di{\displaystyle}
\def\N{\mathbb{N}}
\def\R{\mathbb{R}}
\def\T{\mathbb{T}}
	\def\TK{\mathbb{T}^\kappa}
	\def\Tk{\mathbb{T}_\kappa}
	\def\TKk{\mathbb{T}^\kappa_\kappa}

\def\L{\mathcal{L}}

\def\DD{\Delta}

\def\CC{\mathscr{C}}
\def\CCC{\mathrm{C}}

\def\RS{\mathrm{RS}}
\def\LS{\mathrm{LS}}
\def\RD{\mathrm{RD}}
\def\LD{\mathrm{LD}}

\newcommand{\fonction}[5]{\begin{array}[t]{lrcl}#1 :&#2 &\longrightarrow &#3\\&#4& \longmapsto &#5 \end{array}}

\newcommand{\fonctionsansdef}[3]{#1 : #2 \longrightarrow #3}

\title{Nonshifted calculus of variations on time scales with $\nabla$-differentiable $\sigma$}

\author{Lo\"ic Bourdin\footnote{Laboratoire de Math\'ematiques et de leurs Applications - Pau (LMAP). UMR CNRS 5142. Universit\'e de Pau et des Pays de l'Adour. \texttt{bourdin.l@univ-pau.fr}} }

\date{}

\begin{document}

\maketitle

\begin{abstract}
In calculus of variations on general time scales, an Euler-Lagrange equation of integral form is usually derived in order to characterize the critical points of nonshifted Lagrangian functionals, see \textit{e.g.} [R.A.C. Ferreira and co-authors, \textit{Optimality conditions for the calculus of variations with higher-order delta derivatives}, Appl. Math. Lett., 2011]. In this paper, we prove that the $\nabla$-differentiability of the forward jump operator $\sigma$ is a sharp assumption on the time scale in order to $\nabla$-differentiate this integral Euler-Lagrange equation. This procedure leads to an Euler-Lagrange equation of differential form. Furthermore, from this differential form, we prove a Noether-type theorem providing an explicit constant of motion for Euler-Lagrange equations admitting a symmetry.
\end{abstract}



\tableofcontents

\section{Introduction}\label{section0}
The time scale theory was introduced by S. Hilger in his PhD thesis \cite{hilg} in 1988 in order to unify discrete and continuous analyses. The general idea is to extend classical theories on an arbitrary non-empty closed subset $\T$ of $\R$. Such a subset $\T$ is called a \textit{time scale}. In this paper, it is assumed to be bounded with $a=\min \T$ and $b = \max \T$. Hence, the time scale theory establishes the validity of some results both in the continuous case $\T =[a,b]$ and in the purely discrete case $\T =\{ a = t_0 <  \ldots < t_N = b \}$. Moreover, it also treats more general models involving both continuous and discrete time elements, see \textit{e.g.} \cite{gama,may} for dynamical populations whose generations do not overlap. Another example of application is to consider $\T = \{ a \} \cup \{ a + \lambda^\N \}$ with $0 < \lambda < 1$ allowing the time scale theory to cover the quantum calculus \cite{kac}. 

Since S. Hilger defined the $\DD$- and $\nabla$-derivatives on time scales, many authors have extended to time scales various results from the continuous or discrete standard calculus theory. We refer to the surveys \cite{agar2,agar3,bohn,bohn3} of M. Bohner and co-authors. In the continuous case $\T = [a,b]$, the operators $\DD$ and $\nabla$ coincide with the usual derivative operator $d/dt$, \textit{i.e.} $u^\DD = u^\nabla = \dot{u}$ where $\dot{u}$ is the classical derivative of a function $u$. In the discrete case $\T =\{ a = t_0 < \ldots < t_N = b \}$, $\DD$ is the usual forward Euler approximation of $d/dt$, \textit{i.e.} $\DD u (t_k) = (u(t_{k+1}) - u(t_k))/(t_{k+1}-t_k)$ and similarly, $\nabla$ is the usual backward Euler approximation of $d/dt$, \textit{i.e.} $\nabla u (t_k) = (u(t_{k}) - u(t_{k-1}))/(t_{k}-t_{k-1})$.

\paragraph{Context in shifted calculus of variations.}
The pioneering work on calculus of variations on general time scales is due to M. Bohner in \cite{bohn2}. In particular, he obtains a necessary condition for local optimizers of Lagrangian functionals of type
\begin{equation}\label{eqshiftlagintro}
\L (u ) = \di \int_a^b L(u^\sigma (\tau),u^\DD (\tau), \tau) \; \DD \tau,
\end{equation}
where $L : \R^n \times \R^n \times \T^\kappa \longrightarrow \R, \; (x,v,t) \longmapsto L(x,v,t)$ is a continuous Lagrangian of class $\CC^1$ in its two first variables, $u$ is a $\CCC^{1,\DD}_{\mathrm{rd}}(\T)$-function, $u^\sigma = u \circ \sigma$, $u^\DD$ is the $\DD$-differential of $u$ and $\int \DD \tau$ denotes the Cauchy $\DD$-integral defined in \cite[p.26]{bohn}. We refer to Section~\ref{section11} for the precise definitions. \\

\noindent \textit{Notation: in the whole paper, the function defined by $t \longmapsto \partial L / \partial v (u^\sigma(t),u^\DD(t),t)$, where $u \in \CCC^{1,\DD}_{\mathrm{rd}}(\T)$, is denoted by $\partial L / \partial v(u^\sigma,u^\DD,\cdot)$.} \\

\noindent Precisely, M. Bohner characterizes the critical points of $\L$ as the solutions of the following $\DD \circ \DD$-differential Euler-Lagrange equation, see \cite[Theorem~4.2]{bohn2}:
\begin{equation}\label{eqdiffeldeltadelta}\tag{EL${}^{\DD \circ \DD}_\mathrm{diff}$}
\left[ \dfrac{\partial L}{\partial v} (u^\sigma,u^\DD,\cdot) \right]^\DD (t) = \dfrac{\partial L}{\partial x} (u^\sigma (t),u^\DD (t),t).
\end{equation}
Here, the notation $\DD \circ \DD$ refers to the composition of $\DD$ with itself in the left-hand term of \eqref{eqdiffeldeltadelta}. As it is mentioned in \cite{bohn2}, the result of M. Bohner recovers the usual continuous case $\T = [a,b]$ (where $\sigma$ is the identity) where the critical points of Lagrangian functionals of type
\begin{equation}\label{eqcontintrolag}
\L (u ) = \di \int_a^b L(u (\tau),\dot{u} (\tau), \tau) \; d\tau
\end{equation}
are characterized by the solutions of the well known continuous Euler-Lagrange equation (see \textit{e.g.} \cite[p.12]{arno}) given by
\begin{equation}\label{eqcontintroel}
\dfrac{d}{dt} \left[ \dfrac{\partial L}{\partial v} (u,\dot{u},\cdot) \right] (t) = \dfrac{\partial L}{\partial x} (u(t),\dot{u}(t),t).
\end{equation}
Moreover, the work of M. Bohner in \cite{bohn2} also recovers the following discrete case $\T =\{ a = t_0 <  \ldots < t_N = b \}$  where the critical points of discrete Lagrangian functionals of type
\begin{equation}\label{eqlagshiftintro}
\L (u ) = \di \sum_{k=0}^{N-1} (t_{k+1}-t_k) L(u(t_{k+1}),\DD u(t_k),t_k)
\end{equation}
are characterized by the solutions of the well known discrete Euler-Lagrange equation (see \textit{e.g.} \cite{pete}) given by
\begin{equation}
\DD \left[ \dfrac{\partial L}{\partial v} (u^\sigma,\DD u,\cdot) \right] (t_k) = \dfrac{\partial L}{\partial x} (u(t_{k+1}),\DD u (t_k),t_k).
\end{equation}
In what follows, we will speak about $\L$ (defined in \eqref{eqshiftlagintro}) as a \textit{shifted} Lagrangian functional in reference to the presence of $u^\sigma$ (instead of $u$) in its definition. Note that this characteristic has no consequence on the continuous case but let us mention the presence of $u(t_{k+1})$ (instead of $u(t_k)$) in the discrete case. We will see that this difference is important at the discrete level and \textit{a fortiori} at the time scale one too. In particular, this shifted framework does not cover the variational integrator studied in \cite{lubi,mars}. We refer to the next paragraph for more details.

Since the publication of \cite{bohn2}, the \textit{shifted} calculus of variations is widely investigated in several directions: with double integral \cite{bohn4}, with higher-order $\DD$-derivatives \cite{torr9}, with nonfixed boundary conditions and transversality conditions \cite{hils}, with double integral mixing $\DD$- and $\nabla$-derivatives \cite{torr7}, with higher-order $\nabla$-derivatives \cite{torr11}, etc. We also refer to \cite{hils2,hils5} for \textit{shifted} optimal control problems. Let us mention that \textit{shifted} variational problems are particularly suitable (in comparison with the \textit{nonshifted} ones) because of the emergence of a \textit{shift} in the integration by parts formula on time scales (see \cite[Theorem~1.77 p.28]{bohn}) given by
\begin{equation}
\di \int_a^b u(\tau) \cdot v^\DD(\tau) \; \DD \tau = u(b) \cdot v(b) - u(a) \cdot v(a) - \di \int_a^b u^\DD(\tau) \cdot v^\sigma(\tau) \; \DD \tau.
\end{equation}

\paragraph{Context in nonshifted calculus of variations.}
As mentioned in the previous paragraph, the \textit{shifted} calculus of variations on general time scales developed in \cite{bohn2} does not cover an important discrete calculus of variations. Precisely, recall that the critical points of (nonshifted) discrete Lagrangian functionals of type
\begin{equation}
\L (u ) = \di \sum_{k=0}^{N-1} (t_{k+1}-t_k) L(u(t_k),\DD u(t_k),t_k)
\end{equation}
are characterized by the solutions of the well known discrete Euler-Lagrange equation (see \textit{e.g.} \cite{lubi}) given by
\begin{equation}\label{eqnonshiftdiscrintro}
\nabla \left[ \dfrac{\partial L}{\partial v} (u,\DD u,\cdot) \right] (t_k) = \dfrac{t_{k+1}-t_k}{t_k - t_{k-1}} \; \dfrac{\partial L}{\partial x} (u(t_k),\DD u (t_k),t_k).
\end{equation}
Recall that the above discrete Euler-Lagrange equation~\eqref{eqnonshiftdiscrintro} corresponds to the variational integrator obtained and well studied in \cite{lubi,mars}. In particular, it is an efficient numerical scheme for the continuous Euler-Lagrange equation~\eqref{eqcontintroel} preserving its variational structure and relative properties at the discrete level. In this (nonshifted) discrete case, note the emergence of the composition between the operators $\nabla$ and $\DD$. In numerical analysis, it is well known that such a composition allows to get a larger order of convergence.

Up to our knowledge, only few studies treat on the \textit{nonshifted} calculus of variations on general time scales, see \cite{cres9,torr14,hils6}. In these papers, the critical points of (nonshifted) Lagrangian functionals of type
\begin{equation}\label{eqnonshiftlagintro}
\L (u ) = \di \int_a^b L(u (\tau),u^\DD (\tau), \tau) \; \DD \tau
\end{equation}
are characterized by the solutions of the following $\DD$-integral Euler-Lagrange equation, see \textit{e.g.} \cite[Theorem~4]{hils6}:
\begin{equation}\tag{EL${}^{\DD}_{\text{int}}$}\label{eqintelintro}
\dfrac{\partial L}{\partial v} (u(t),u^\DD(t),t) = \di \int_a^{\sigma (t)}  \dfrac{\partial L}{\partial x} (u(\tau),u^\DD(\tau),\tau) \; \DD \tau + c.
\end{equation}
The objective of this paper is to $\nabla$-differentiate \eqref{eqintelintro} in order to get a $\nabla \circ \DD$-differential Euler-Lagrange equation of type
\begin{equation}\label{eqdiffelintro}
\left[ \dfrac{\partial L}{\partial v} (u,u^\DD,\cdot) \right]^\nabla (t) =  \omega (t) \dfrac{\partial L}{\partial x} (u(t),u^\DD(t),t)
\end{equation}
that encompasses both the classical continuous case~\eqref{eqcontintroel} and the (nonshifted) discrete case~\eqref{eqnonshiftdiscrintro}. Towards this goal, a difficulty emerges due to the presence of $\sigma$ in the upper bound of the $\DD$-integral in \eqref{eqintelintro}. In fact, we will exhibit a counter-example (see Example~\ref{ex1}) showing that we cannot $\nabla$-differentiate \eqref{eqintelintro} on general time scales with a general Lagrangian $L$. It leads us to consider a subclass of time scales. Precisely, we prove that the $\nabla$-differentiability of $\sigma$ is a sharp assumption on the time scale in order to $\nabla$-differentiate~\eqref{eqintelintro} and to obtain a $\nabla \circ \DD$-differential Euler-Lagrange equation of type~\eqref{eqdiffelintro}. Moreover, in such a case, an interesting phenomena is the direct emergence of this assumption in \eqref{eqdiffelintro} since we prove that $\omega=\sigma^\nabla$. 

In this paper, we study the consequences of the $\nabla$-differentiability of $\sigma$ on the structure of $\T$. In particular, we will see how this assumption allows us to apply $\nabla$ on \eqref{eqintelintro} in order to obtain \eqref{eqdiffelintro} with $\omega = \sigma^\nabla$. Note that the $\nabla$-differentiability of $\sigma$ is not a loss of generality since it is satisfied in the continuous case $\T =[a,b]$ (with $\sigma^\nabla =1$) and in the discrete case $\T =\{ a = t_0 < \ldots < t_N = b \}$ (with $\sigma^\nabla (t_k) = (t_{k+1}-t_k)/(t_k - t_{k-1})$). In particular, note that our main result recovers both the usual continuous case \eqref{eqcontintroel} and the nonshifted discrete case \eqref{eqnonshiftdiscrintro}.

\paragraph{Derivation of Noether-type results.}
In \textit{shifted} calculus of variations, we refer to the paper \cite{torr8} studying the existence of a constant of motion for $\DD \circ \DD$-differential Euler-Lagrange equations \eqref{eqdiffeldeltadelta}. We refer to \cite{torr12} for a similar study with $\nabla$-derivatives. The common strategy is to generalize the celebrated Noether's theorem \cite{noet2,noet} to time scales. Precisely, under invariance assumption on the Lagrangian $L$, authors prove that a conservation law can be obtained. 

In \textit{nonshifted} calculus of variations, the nondifferential form of \eqref{eqintelintro} is an obstruction in order to develop the same strategy. A direct application of our main result is then to provide a Noether-type theorem based on the differential form \eqref{eqdiffelintro}. This will be done in Section~\ref{section3}.

\begin{remark}
For sake of completeness of this introduction, we mention that an Euler-Lagrange equation of \textit{differential} form in the nonshifted case is obtained in \cite[Remark~4]{hils6}. Precisely, the author characterizes the critical points of $\L$ (defined in \eqref{eqnonshiftlagintro}) as the solutions of the following differential Euler-Lagrange equation:
\begin{equation}
\left[ \dfrac{\partial L}{\partial v}(u,u^\DD,\cdot) - \mu \dfrac{\partial L}{\partial x}(u,u^\DD,\cdot) \right]^\DD (t) = \dfrac{\partial L}{\partial x}(u(t),u^\DD(t),t).
\end{equation}
The advantage of this result is to be valid on every time scale. Nevertheless, this differential form does not directly coincide with the usual discrete Euler-Lagrange equation \eqref{eqnonshiftdiscrintro} and the obtaining of a Noether-type result from this differential form remains an open problem. These observations both give a particular interest to the $\nabla$-differentiation of~\eqref{eqintelintro} and to the $\nabla \circ \DD$-differential formulation~\eqref{eqdiffelintro}.
\end{remark}

\begin{remark}
Finally, it has to be noted that our whole study is made in terms of Lagrangian functionals involving $\DD$-integral and $\DD$-derivative. However, thanks to the duality principle introduced in \cite{capu}, all results can be analogously derived for Lagrangian functionals involving $\nabla$-integral and $\nabla$-derivative. This will be done in Section~\ref{section4}.
\end{remark}

\paragraph{Organization of the paper.}
We first give basic recalls on time scale calculus in Section~\ref{section11} and on nonshifted calculus of variations on general time scales in Section~\ref{section12}. Section~\ref{section13} is devoted to our main result (Theorem~\ref{thmmain}). A study on time scales with continuous $\sigma$ is provided in Section~\ref{section21} and with $\nabla$-differentiable $\sigma$ in Sections~\ref{section22} and \ref{section23}. The results obtained in Section~\ref{section2} are instrumental in order to prove Theorem~\ref{thmmain}. We prove a Noether-type theorem (Theorem~\ref{thmnoether}) in Section~\ref{section3}. In Section~\ref{section4}, we conclude this paper with the analogous results for nonshifted calculus of variations defined in terms of $\nabla$-integral and $\nabla$-derivative.

\section{Nonshifted calculus of variations on time scales with $\nabla$-differentiable $\sigma$}\label{section1}
In this paper, $\N$ denotes the set of nonnegative integers, $\N^*$ denotes the set of positive integers and $\R^+$ denotes the set of nonnegative reals. We denote by $\T$ a bounded time scale with $a = \min (\T)$, $b = \max (\T)$ and $\mathrm{card} (\T) \geq 3$. In Section~\ref{section11}, we give basic recalls on time scale calculus.

Section~\ref{section12} is devoted to recalls on nonshifted calculus of variations on general time scales developed in \cite{cres9,torr14,hils6}. In particular, the $\DD$-integral Euler-Lagrange equation~\eqref{eqintelintro} is given as a necessary condition for local optimizers of nonshifted Lagrangian functionals, see Proposition~\ref{propintel}.

In Section~\ref{section13}, under the assumption of $\nabla$-differentiability of $\sigma$, our main result provides a $\nabla \circ \DD$-differential Euler-Lagrange equation of type \eqref{eqdiffelintro} as a necessary condition for local optimizers of nonshifted Lagrangian functionals, see \eqref{eqdiffel} in Theorem~\ref{thmmain}. We also prove that this assumption is sharp, see Example~\ref{ex1}.

\subsection{Basic recalls on time scale calculus}\label{section11}
We refer to the surveys \cite{agar2,agar3,bohn,bohn3} for more details on time scale calculus. The backward and forward jump operators $\rho, \sigma : \T \longrightarrow \T$ are respectively defined by
\begin{equation}
\forall t \in \T, \; \rho (t) = \sup \{ s \in \T, \; s < t \} \; \text{and} \; \sigma (t) = \inf \{ s \in \T, \; s > t \},
\end{equation}
where we put $\sup \emptyset = a$ and $\inf \emptyset = b$. A point $t \in \T$ is said to be right-scattered (resp. left-scattered) if $\sigma (t) > t$ (resp. $\rho (t) < t$). A point $t \in \T$ is said to be right-dense (resp. left-dense) if $t \neq b$ and $\sigma (t) = t$ (resp. $t \neq a$ and $\rho (t) = t$). Let $\RS$ (resp. $\LS$, $\RD$ and $\LD$) denote the set of all right-scattered (resp. left-scattered, right-dense and left-dense) points of $\T$. Note that $a \notin \LD$ and $b \notin \RD$. The graininess (resp. backward graininess) function $\fonctionsansdef{\mu}{\T}{\R^+}$ (resp. $\fonctionsansdef{\nu}{\T}{\R^+}$) is defined by $\mu(t) = \sigma (t) -t$ (resp. $\nu(t) = t- \rho (t)$) for any $t \in \T$.

We set $\TK = \T \backslash ]\rho(b),b]$, $\Tk = \T \backslash [a,\sigma(a)[$ and $\TKk = \TK \cap \Tk$. Note that $\TKk \neq \emptyset$ since $\mathrm{card} (\T) \geq 3$. Let us recall the usual definitions of $\DD$- and $\nabla$-differentiability. A function $\fonctionsansdef{u}{\T}{\R^n}$, where $n \in \N^*$, is said to be $\DD$-differentiable at $t \in \TK$ (resp. $\nabla$-differentiable at $t \in \Tk$) if the following limit exists in $\R^n$:
\begin{equation}
\lim\limits_{\substack{s \to t \\ s \neq \sigma (t) }} \dfrac{u(\sigma(t))-u(s)}{\sigma(t) -s} \; \left( \text{resp.} \; \lim\limits_{\substack{s \to t \\ s \neq \rho (t) }} \dfrac{u(s)-u(\rho (t))}{s-\rho(t)} \right).
\end{equation}
In such a case, this limit is denoted by $u^\DD (t)$ (resp. $u^\nabla (t)$). Let us recall the following results on $\DD$-differentiability, see \cite[Theorem~1.16 p.5]{bohn} and \cite[Corollary~1.68 p.25]{bohn}. The analogous results for $\nabla$-differentiability are also valid.

\begin{proposition}\label{proprappeldelta}
Let $\fonctionsansdef{u}{\T}{\R^n}$ and $t \in \TK$. The following properties hold:
\begin{enumerate}
\item if $u$ is $\DD$-differentiable at $t$, then $u$ is continuous at $t$.
\item if $t \in \RS$ and if $u$ is continuous at $t$, then $u$ is $\DD$-differentiable at $t$ with
\begin{equation}
u^\DD (t) = \dfrac{u(\sigma(t))-u(t)}{\mu(t)}.
\end{equation}
\item if $\sigma (t)=t$, then $u$ is $\DD$-differentiable at $t$ if and only if the following limit exists in $\R^n$:
\begin{equation}
\lim\limits_{\substack{s \to t \\ s \neq t}} \dfrac{u(t)-u(s)}{t-s}.
\end{equation}
In such a case, this limit is equal to $u^\DD (t)$.
\end{enumerate}
\end{proposition}

\begin{proposition}\label{proprappeldelta2}
Let $\fonctionsansdef{u}{\T}{\R^n}$. Then, $u$ is $\DD$-differentiable on $\TK$ with $u^\DD = 0$ if and only if there exists $c \in \R^n$ such that $u(t) =c$ for every $t \in \T$.
\end{proposition}

From Proposition~\ref{proprappeldelta} and for every $t\in\RS$, note that a function $u$ is $\DD$-differentiable at $t$ if and only if $u$ is continuous at $t$. Still from Proposition~\ref{proprappeldelta}, note that every $\DD$-differentiable function on $\TK$ is continuous on $\T$. 

Recall that a function $u$ is said to be rd-continuous on $\T$ if it is continuous at every $t \in \RD$ and if it admits a left-sided limit at every $t \in \LD$, see \cite[Definition~1.58 p.22]{bohn}. We respectively denote by $\CCC^0_{\mathrm{rd}}(\T)$ and $\CCC^{1,\DD}_{\mathrm{rd}}(\T)$ the functional spaces of rd-continuous functions on $\T$ and of $\DD$-differentiable functions on $\TK$ with rd-continuous $\DD$-derivative. Recall the following results, see \cite[Theorem~1.60 p.22]{bohn}:
\begin{itemize}
\item $\sigma$ is rd-continuous.
\item if $u \in \CCC^0_{\mathrm{rd}}(\T)$, the composition $u^\sigma = u \circ \sigma$ is rd-continuous.
\item if $u \in \CCC^0_{\mathrm{rd}}(\T)$, the composition $f \circ u$ with any continuous function $f$ is rd-continuous.
\end{itemize}
Let us denote by $\int \DD \tau$ the Cauchy $\DD$-integral defined in \cite[p.26]{bohn}. For every $u \in \CCC^0_{\mathrm{rd}}(\TK)$, recall that the function $U$, defined by $U(t) = \int_a^t u(\tau) \DD \tau$ for every $t \in \T$, is the unique $\DD$-antiderivative of $u$ (in the sense that $U^\DD = u$ on $\TK$) vanishing at $t=a$, see \cite[Theorem~1.74 p.27]{bohn}. In particular, we have $U \in \CCC^{1,\DD}_{\mathrm{rd}}(\T)$. 

\subsection{Recalls on nonshifted calculus of variations on general time scales}\label{section12}
In this section, we recall some results on nonshifted calculus of variations on general time scales
provided in \cite{cres9,torr14,hils6}. Let $L$ be a Lagrangian \textit{i.e.} a continuous map of class $\CC^1$ in its two first variables
\begin{equation}
\fonction{L}{\R^n \times \R^n \times \TK}{\R}{(x,v,t)}{L(x,v,t)}
\end{equation}
and let $\L$ be the following (nonshifted) Lagrangian functional:
\begin{equation}
\fonction{\L}{\mathrm{C}^{1,\DD}_{\mathrm{rd}}(\T)}{\R}{u}{\di \int_a^b L(u(\tau),u^\DD(\tau),\tau) \; \DD \tau.}
\end{equation}

In this section, our aim is to give a necessary condition for local optimizers of $\L$ (with or without boundary conditions at $t=a$ and $t=b$). In this way, we introduce the following notions and notations:
\begin{itemize}
\item $\mathrm{C}^{1,\DD}_{\mathrm{rd},0} (\T) = \{ w \in \mathrm{C}^{1,\DD}_{\mathrm{rd}} (\T), \; w(a)= w(b) = 0\}$ is called the set of \textit{variations} of $\L$.
\item $u \in \mathrm{C}^{1,\DD}_{\mathrm{rd}}(\T)$ is said to be a \textit{critical point} of $\L$ if $D\L(u)(w) = 0$ for every $w \in \mathrm{C}^{1,\DD}_{\mathrm{rd},0}(\T)$. Let us precise that $D\L(u)(w)$ denotes the G\^ateaux-differential of $\L$ at $u$ in direction $w$.
\end{itemize}
In particular, if $u$ is a local optimizer of $\L$, then $u$ is a critical point of $\L$. Finally, let us recall the following characterization of the critical points of $\L$, see \cite[Theorem~11]{cres9}, \cite[Corollary~1]{torr14} or \cite[Theorem~4]{hils6}.

\begin{proposition}\label{propintel}
Let $u \in \mathrm{C}^{1,\DD}_{\mathrm{rd}}(\T)$. Then, $u$ is a critical point of $\L$ if and only if there exists $c \in \R^n$ such that
\begin{equation}\label{eqintel}\tag{EL${}^{\DD}_{\text{int}}$}
\dfrac{\partial L}{\partial v} (u(t),u^\DD(t),t) = \di \int_a^{\sigma (t)}  \dfrac{\partial L}{\partial x} (u(\tau),u^\DD(\tau),\tau) \; \DD \tau + c,
\end{equation}
for every $t \in \TK$.
\end{proposition}

Hence, Proposition~\ref{propintel} provides a necessary condition for local optimizers of $\L$. Precisely, if $u$ is a local optimizer of $\L$, then there exists $c \in \R^n$ such that $u$ satisfies the $\DD$-integral Euler-Lagrange equation~\eqref{eqintel}. We refer to Example~\ref{ex1} in Section~\ref{section13} for an application of Proposition~\ref{propintel}.

\subsection{Main result}\label{section13}

In this paper, we aim to $\nabla$-differentiate the $\DD$-integral Euler-Lagrange equation~\eqref{eqintel} in order to get a $\nabla \circ \DD$-differential one of type \eqref{eqdiffelintro}. Precisely, we prove the following result under the assumption of $\nabla$-differentiability of $\sigma$.

\begin{theorem}[Main result]\label{thmmain}
Let us assume that $\sigma$ is $\nabla$-differentiable on $\Tk$ and let $u \in \mathrm{C}^{1,\DD}_{\mathrm{rd}}(\T)$. Then, $u$ is a critical point of $\L$ if and only if $u$ is a solution of the following $\nabla \circ \DD$-differential Euler-Lagrange equation:
\begin{equation}\label{eqdiffel}\tag{EL${}^{\nabla \circ \DD}_{\text{diff}}$}
\left[ \dfrac{\partial L}{\partial v} (u,u^\DD,\cdot) \right]^\nabla (t) =  \sigma^\nabla (t) \dfrac{\partial L}{\partial x} (u(t),u^\DD(t),t) ,
\end{equation}
for every $t \in \TKk$.
\end{theorem}

\begin{proof}
We refer to Propositions~\ref{proprappeldelta2} and \ref{propintel} and Corollary~\ref{cor1} in Section~\ref{section23}.
\end{proof}

Note that this result encompasses both the usual continuous and discrete Euler-Lagrange equations given by \eqref{eqcontintroel} and \eqref{eqnonshiftdiscrintro} in Introduction. Indeed, as it is mentioned in Example~\ref{ex3} in Section~\ref{section22}, the following properties are satisfied:
\begin{itemize}
\item if $\T = [a,b]$, $\sigma$ is $\nabla$-differentiable on $\Tk$ with $\sigma^\nabla = 1$.
\item if $\T =\{ a = t_0 < \ldots < t_N = b \}$, $\sigma$ is $\nabla$-differentiable on $\Tk$ with $\sigma^\nabla = \mu / \nu$.
\end{itemize}

Let us prove, from the following simple example, that the assumption of $\nabla$-differentiability of $\sigma$ is sharp for the validity of Theorem~\ref{thmmain} when the Lagrangian $L$ is general (\textit{i.e.} not specified).

\begin{example}\label{ex1}
Let us consider $n=1$, $L(x,v,t)= x+v^2 /2$ and $u \in \CCC^{1,\DD}_\mathrm{rd}(\T)$ defined by $u(t)=\int_a^t \sigma (\tau) \DD\tau$ for every $t \in \T$. Since $u$ satisfies \eqref{eqintel} with $c=a$, we conclude that $u$ is a critical point of $\L$, see Proposition~\ref{propintel}. However, note that $\partial L / \partial v (u,u^\DD,\cdot) = u^\DD = \sigma$ and consequently, the $\nabla \circ \DD$-differential Euler-Lagrange equation~\eqref{eqdiffel} has no sense if $\sigma$ is not $\nabla$-differentiable.
\end{example}

Nevertheless, the following example proves that the $\nabla$-differentiability of $\sigma$ is not necessary for some Lagrangian (\textit{e.g.} independent of the variable $x$).

\begin{example}\label{ex11}
Let us consider $n=1$, $L(x,v,t)= v^2 /2$ and $u \in \CCC^{1,\DD}_\mathrm{rd}(\T)$ defined by $u(t)=t$ for every $t \in \T$. Since $u$ satisfies \eqref{eqintel} with $c=1$, we conclude that $u$ is a critical point of $\L$, see Proposition~\ref{propintel}. Note that $u$ also satisfies \eqref{eqdiffel} even if $\sigma$ is not $\nabla$-differentiable. This phenomena comes from $\partial L / \partial x =0$ and consequently \eqref{eqintel} is independent of $\sigma$ and it can be $\nabla$-differentiated.
\end{example}

However, it has to be noted that the independence of the Lagrangian $L$ with respect to the variable $x$ is a very restrictive assumption while the $\nabla$-differentiability of $\sigma$ is not. We refer to Example~\ref{ex2} for time scales with continuous and noncontinuous $\sigma$. We respectively refer to Examples~\ref{ex3} and \ref{ex4} for time scales with $\nabla$-differentiable and non-$\nabla$-differentiable $\sigma$.

%

\section{Time scales with $\nabla$-differentiable $\sigma$}\label{section2}
In Section~\ref{section21}, we study the consequences of the continuity of $\sigma$. Then, we study the consequences of the $\nabla$-differentiability of $\sigma$ in Section~\ref{section22}. From these preliminaries, we prove the most important results of this section (Proposition~\ref{thm1} and some corollaries) in Section~\ref{section23}. In particular, Corollary~\ref{cor1} is instrumental to prove our main result (Theorem~\ref{thmmain}).

\subsection{Continuity of $\sigma$}\label{section21}
Let us prove the following characterizations of the continuity of $\sigma$ at a given point $t \in \Tk$.

\begin{lemma}\label{propcont}
Let $t \in \Tk$. The following properties are equivalent:
\begin{enumerate}
\item $\sigma$ is continuous at $t$;
\item $\sigma \circ \rho (t) = t$;
\item $t \notin \RS \cap \LD$.
\end{enumerate}
\end{lemma}

\begin{proof}
Let us prove that $1.$ implies $2.$. If $t=a \in \Tk$, then $a \in \RD$. Then, $\sigma \circ \rho (a) = \sigma (a) = a$. If $t \neq a$, let us assume by contradiction that $\sigma \circ \rho (t) \neq t$. Necessarily, we have $t \in \RS \cap \LD$. Then, let $(s_k) \subset \T$ be a sequence such that $s_k < t$ for any $k \in \N$ and $s_k \to t$. Thus, we have $\sigma (s_k) < t < \sigma (t)$ for any $k \in \N$ and consequently, $(\sigma (s_k))$ does not tend to $\sigma (t)$. This is a contradiction with the continuity of $\sigma$ at $t$.

Let us prove that $2.$ implies $3.$. If $ t \in  \RS \cap \LD$, then $\sigma \circ \rho (t) = \sigma (t) \neq t$.

Let us prove that $3.$ implies $1.$. By contradiction, let us assume that $\sigma$ is not continuous at $t$. As a consequence, there exist $\varepsilon > 0$ and a monotone sequence $(s_k) \subset \T$ such that $s_k \to t$ and $\vert \sigma (t) - \sigma (s_k) \vert \geq \varepsilon$ for every $k \in \N$. Firstly, let us assume that $(s_k)$ is decreasing. Then, $t \in \RD$ and we have $t < s_k < s_{k-1}$. As a consequence, $t = \sigma (t) < \sigma (s_k) \leq s_{k-1}$ for any $k \in \N^*$. It is a contradiction since $s_{k-1} \to t$. Secondly, let us assume that $(s_k)$ is increasing. As a consequence, $t \in \LD$ and then $t \notin \RS$ (see 3.) \textit{i.e.} $\sigma(t) = t$. Finally, we have $s_{k-1} < s_k < t$ and then, $s_{k-1} < \sigma (s_k) < t = \sigma(t)$ for any $k \in \N^*$. It is a contradiction since $s_{k-1} \to t$. In both cases, we have obtained a contradiction.
\end{proof}

Note that $\sigma$ is always continuous at $a$. Indeed, if $a \in \RD$, then $a \in \Tk$, $a \notin \RS \cap \LD$ and Lemma~\ref{propcont} concludes. If $a \in \RS$, then $a$ is isolated and thus, $\sigma$ is continuous at $a$. This remark and Lemma~\ref{propcont} lead to the following proposition.

\begin{proposition}\label{propsigmacont1}
$\sigma$ is continuous on $\T$ if and only if $\RS \cap \LD = \emptyset$.
\end{proposition}
%

A similar remark is already done in \cite[Example~1.55]{bohn}. Let us see some examples and counter-examples.

\begin{example}\label{ex2}
\begin{enumerate}
\item If $\T =[a,b]$, $\sigma$ is continuous on $\T$.
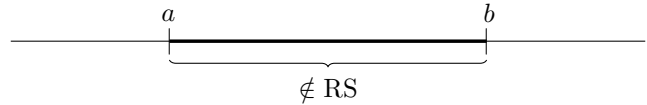
\begin{center}
\begin{tikzpicture}
	\begin{scope}[scale=4,shift={(0,0)},x=15pt,y=15pt]
		\draw (-2,0)--(2,0);
		\draw[line width=1.3pt] (-1,0)--(1,0);
		\node at (-1,0) {$|$};
		\node at (1,0) {$|$};
		\draw (-1.11,0.26) node[anchor=north west] {$a$};
		\draw (0.92,0.30) node[anchor=north west] {$b$};
		\draw[decorate,decoration={brace,raise=0.25cm}] (1,0) -- (-1,0) node[below=10,pos=0.5] {$\notin \RS$};
	\end{scope}
\end{tikzpicture}
\end{center}
\item If $\T = \{ a = t_0 < \ldots < t_N = b \}$, $\sigma$ is continuous on $\T$.
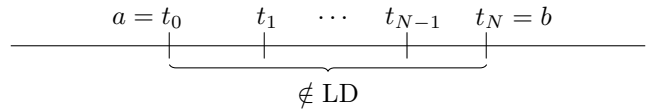
\begin{center}
\begin{tikzpicture}
	\begin{scope}[scale=4,shift={(0,0)},x=15pt,y=15pt]
		\draw (-2,0)--(2,0);
		\node at (-1,0) {$|$};
		\node at (-0.4,0) {$|$};
		\node at (0.5,0) {$|$};
		\node at (1,0) {$|$};
		\node at (0.03,0.18) {$\ldots$};
		\draw (-1.42,0.3) node[anchor=north west] {$a=t_0$};
		\draw (-0.51,0.3) node[anchor=north west] {$t_1$};
		\draw (0.3,0.3) node[anchor=north west] {$t_{N-1}$};
		\draw (0.87,0.31) node[anchor=north west] {$t_N = b$};
		\draw[decorate,decoration={brace,raise=0.25cm}] (1,0) -- (-1,0) node[below=10,pos=0.5] {$\notin \LD$};
	\end{scope}
\end{tikzpicture}
\end{center}
\item If $\T = \{ 0,1 \} \cup [2,3]$, $\sigma$ is continuous on $\T$. 
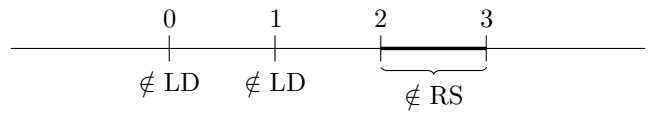
\begin{center}
\begin{tikzpicture}
	\begin{scope}[scale=4,shift={(0,0)},x=15pt,y=15pt]
		\draw (-2,0)--(2,0);
		\draw[line width=1.3pt] (1/3,0)--(1,0);
		\node at (-1,0) {$|$};
		\node at (-1,-0.21) {$\notin \LD$};
		\node at (-1/3,-0.21) {$\notin \LD$};
		\node at (-1/3,0) {$|$};
		\node at (1/3,0) {$|$};
		\node at (1,0) {$|$};
		\draw (-1.1,0.3) node[anchor=north west] {$0$};
		\draw (-0.43,0.3) node[anchor=north west] {$1$};
		\draw (0.23,0.3) node[anchor=north west] {$2$};
		\draw (0.9,0.30) node[anchor=north west] {$3$};
		\draw[decorate,decoration={brace,raise=0.25cm}] (1,0) -- (1/3,0) node[below=10,pos=0.5] {$\notin \RS$};
	\end{scope}
\end{tikzpicture}
\end{center}
\item If $\T = [-1,0] \cup \{ 1/k, \; k \in \N^* \}$, $\sigma$ is continuous on $\T$.
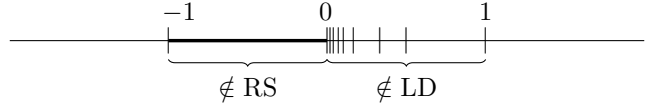
\begin{center}
\begin{tikzpicture}
	\begin{scope}[scale=4,shift={(0,0)},x=15pt,y=15pt]
		\draw (-2,0)--(2,0);
		\draw[line width=1.3pt] (-1,0)--(0,0);
		\node at (-1,0) {$|$};
		\node at (0,0) {$|$};
		\node at (1/2,0) {$|$};
		\node at (1/3,0) {$|$};
		\node at (1/6,0) {$|$};
		\node at (5/48,0) {$|$};
		\node at (7/96,0) {$|$};
		\node at (1/24,0) {$|$};
		\node at (1/48,0) {$|$};
		\node at (1,0) {$|$};
		\draw (-1.1,0.3) node[anchor=north west] {$-1$};
		\draw (-0.11,0.3) node[anchor=north west] {$0$};
		\draw (0.9,0.30) node[anchor=north west] {$1$};
		\draw[decorate,decoration={brace,raise=0.25cm}] (1,0) -- (0,0) node[below=10,pos=0.5] {$\notin \LD$};
		\draw[decorate,decoration={brace,raise=0.25cm}] (0,0) -- (-1,0) node[below=10,pos=0.5] {$\notin \RS$};
	\end{scope}
\end{tikzpicture}
\end{center}
\item If $\T = [0,1] \cup [2,3]$, $\sigma$ is not continuous at $1 \in \RS \cap \LD$.
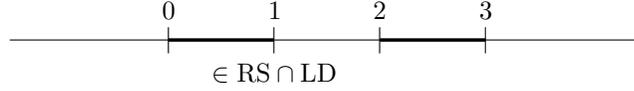
\begin{center}
\begin{tikzpicture}
	\begin{scope}[scale=4,shift={(0,0)},x=15pt,y=15pt]
		\draw (-2,0)--(2,0);
		\draw[line width=1.3pt] (-1/3,0)--(-1,0);
		\draw[line width=1.3pt] (1/3,0)--(1,0);
		\node at (-1,0) {$|$};
		\node at (-1/3,-0.21) {$\in \RS \cap \LD$};
		\node at (-1/3,0) {$|$};
		\node at (1/3,0) {$|$};
		\node at (1,0) {$|$};
		\draw (-1.1,0.3) node[anchor=north west] {$0$};
		\draw (-0.43,0.3) node[anchor=north west] {$1$};
		\draw (0.23,0.3) node[anchor=north west] {$2$};
		\draw (0.9,0.30) node[anchor=north west] {$3$};
	\end{scope}
\end{tikzpicture}
\end{center}
\item If $\T$ is the usual Cantor set (see \cite[Example~1.47 p.18]{bohn}), $\sigma$ is not continuous at $1/3 \in \RS \cap \LD$. 
\end{enumerate}
\end{example}

\begin{remark}
Let us give a short discussion on the notion of \textit{regular} time scale originally introduced in \cite[Definition~9]{gurs}. We refer to \cite{bart,torr10} for other applications of this notion. Recall that $\T$ is said to be regular if for every $t \in \T$, $\sigma \circ \rho (t) = \rho \circ \sigma (t) = t$. In particular, for every bounded regular time scale (containing at least two elements), $a$ is necessarily right-dense and $b$ is necessarily left-dense, see \cite[Proposition~10]{gurs}. Hence, all finite time scales (containing at least two elements) are not regular. The regularity of a time scale is then an assumption relatively restrictive. 

Consequently, we suggest the introduction of the following weakened notion: a time scale is said to be quasi-regular if $\sigma$ and $\rho$ are continuous on $\T$. Hence, a time scale is quasi-regular if and only if $\sigma \circ \rho (t)= t$ for every $t \in \Tk$ and $\rho \circ \sigma (t) = t$ for every $t \in \TK$. Such a weakened notion allows to cover the finite time scales and to preserve the essence of the initial notion.
\end{remark}

\subsection{$\nabla$-differentiability of $\sigma$}\label{section22}

From Lemma~\ref{propcont} and from the nabla version of Proposition~\ref{proprappeldelta}, we derive the following result.

\begin{proposition}\label{propcont2}
The following properties are satisfied:
\begin{enumerate}
\item if $\sigma$ is continuous at $t \in \LS$, then $\sigma$ is directly $\nabla$-differentiable at $t$ with $\sigma^\nabla (t) = \mu (t) / \nu (t)$.
\item if $\sigma$ is continuous on $\T$, then $\sigma$ is $\nabla$-differentiable on $\Tk$ if and only if for every $t \in \Tk$ such that $\rho (t) = t$, the following limit exists in $\R$:
\begin{equation}
\lim\limits_{\substack{s \to t \\ s \neq t }} \dfrac{\sigma (s) -t}{s-t}.
\end{equation}
In such a case, this limit is equal to $\sigma^\nabla (t)$.
\end{enumerate}
\end{proposition}

\begin{proof}
Let us prove the first point. From the nabla version of Proposition~\ref{proprappeldelta} and since $\sigma$ is continuous at $t \in \LS \subset \Tk$, $\sigma$ is directly $\nabla$-differentiable at $t$ with
\begin{equation}
\sigma^\nabla (t) = \dfrac{\sigma(t) - \sigma (\rho (t))}{\nu (t)} = \dfrac{\sigma(t) - t}{\nu (t)} = \dfrac{\mu(t)}{\nu (t)},
\end{equation}
since $\sigma \circ \rho (t) =t$ from Lemma~\ref{propcont}.

Let us prove the second point. Since $\sigma$ is continuous on $\T$, $\sigma$ is directly $\nabla$-differentiable at every $t \in \LS$ from the first point. Consequently, $\sigma$ is $\nabla$-differentiable on $\Tk$ if and only if $\sigma$ is $\nabla$-differentiable at every $t \in \Tk$ such that $\rho (t) = t$ \textit{i.e.} if and only if for every $t \in \Tk$ such that $\rho (t) = t$, the following limit exists in $\R$:
\begin{equation}
\lim\limits_{\substack{s \to t \\ s \neq \rho (t) }} \dfrac{\sigma (s) - \sigma (\rho(t))}{s-\rho(t)} = \lim\limits_{\substack{s \to t \\ s \neq t }} \dfrac{\sigma (s) - \sigma (t)}{s-t}.
\end{equation}
To conclude, it is sufficient to note that if $t=a \in \Tk$, then $a \in \RD$ and $\sigma (a)=a$ and if $t \neq a$, then $t \in \LD$ and the continuity of $\sigma$ implies that $t \notin \RS$ (see Proposition~\ref{propsigmacont1}) \textit{i.e.} $\sigma (t) = t$. The proof is complete.
\end{proof}

Let us give some examples of time scale with $\nabla$-differentiable $\sigma$.

\begin{example}\label{ex3}
\begin{enumerate}
\item If $\T =[a,b]$, $\sigma$ is $\nabla$-differentiable on $\Tk$ with $\sigma^\nabla = 1$.
\item If $\T = \{ a = t_0 < \ldots < t_N = b \}$, $\sigma$ is $\nabla$-differentiable on $\Tk$ with $\sigma^\nabla = \mu / \nu$.
\item If $\T = \{ 0 \} \cup \{ z_k, \; k \in \N \}$ where $(z_k)$ is a decreasing positive sequence tending to $0$ and if $\lim_{k \to \infty} z_{k-1}/z_k$ exists (denoted by $\ell$), then $\sigma$ is $\nabla$-differentiable on $\Tk$. In particular, we have $\sigma^\nabla (0) = \ell$. Indeed, let $(s_k) \subset \T$ be a positive sequence tending to $0$. Then, for every $k \in \N$, there exists $p_k \in \N$ such that $s_k = z_{p_k}$. Since $s_k \rightarrow 0$, we have $p_k \rightarrow +\infty$. Finally, we obtain
\begin{equation}
\lim\limits_{k \to \infty} \dfrac{\sigma (s_k) - 0}{s_k-0} = \lim\limits_{k \to \infty} \dfrac{z_{p_k -1}}{z_{p_k}} = \ell.
\end{equation}
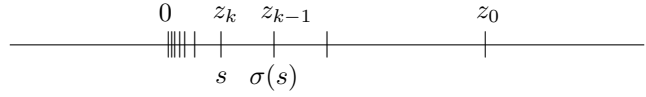
\begin{center}
\begin{tikzpicture}
	\begin{scope}[scale=4,shift={(0,0)},x=15pt,y=15pt]
		\draw (-2,0)--(2,0);
		\node at (-1,0) {$|$};
		\node at (0,0) {$|$};
		\node at (-1/3,0) {$|$};
		\node at (-2/3,0) {$|$};
		\node at (-5/6,0) {$|$};
		\node at (-43/48,0) {$|$};
		\node at (-89/96,0) {$|$};
		\node at (-23/24,0) {$|$};
		\node at (-47/48,0) {$|$};
		\node at (1,0) {$|$};
		\draw (-1.12,0.32) node[anchor=north west] {$0$};
		\draw (-0.78,0.3) node[anchor=north west] {$z_k$};
		\draw (-0.48,0.3) node[anchor=north west] {$z_{k-1}$};
		\draw (-2/3,-0.30) node[anchor=south] {$s$};
		\draw (-1/3,-0.35) node[anchor=south] {$\sigma (s)$};
		\draw (0.88,0.30) node[anchor=north west] {$z_0$};
	\end{scope}
\end{tikzpicture}
\end{center}
\item Application: if $\T = \{ 0 \} \cup \{ 1/r^k, \; k \in \N \}$ with $r > 1$, then $\sigma$ is $\nabla$-differentiable on $\Tk$. In particular, we have $\sigma^\nabla (0) = r$.
\item Similarly to $3.$, we can prove that if $\T = \{ 0 \} \cup \{ z_k, \; k \in \N \}$ where $(z_k)$ is an increasing negative sequence tending to $0$ and if $\lim_{k \to \infty} z_{k+1}/z_k$ exists (denoted by $\ell$), then $\sigma$ is $\nabla$-differentiable on $\Tk$. In particular, we have $\sigma^\nabla (0) = \ell$.
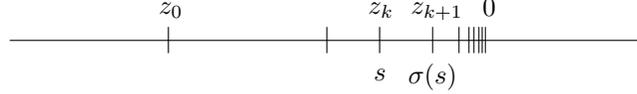
\begin{center}
\begin{tikzpicture}
	\begin{scope}[scale=4,shift={(0,0)},x=15pt,y=15pt]
		\draw (-2,0)--(2,0);
		\node at (-1,0) {$|$};
		\node at (0,0) {$|$};
		\node at (1/3,0) {$|$};
		\node at (2/3,0) {$|$};
		\node at (5/6,0) {$|$};
		\node at (43/48,0) {$|$};
		\node at (89/96,0) {$|$};
		\node at (23/24,0) {$|$};
		\node at (47/48,0) {$|$};
		\node at (1,0) {$|$};
		\draw (-1.12,0.30) node[anchor=north west] {$z_0$};
		\draw (0.47,0.3) node[anchor=north west] {$z_{k+1}$};
		\draw (0.20,0.3) node[anchor=north west] {$z_{k}$};
		\draw (2/3,-0.37) node[anchor=south] {$\sigma (s)$};
		\draw (1/3,-0.31) node[anchor=south] {$s$};
		\draw (0.92,0.32) node[anchor=north west] {$0$};
	\end{scope}
\end{tikzpicture}
\end{center}
\item Application: if $\T = \{ 0 \} \cup \{ -1/r^k, \; k \in \N \}$ with $r > 1$, then $\sigma$ is $\nabla$-differentiable on $\Tk$. In particular, we have $\sigma^\nabla (0) = 1/r$.
\item Similarly to $3.$, we can prove that if $\T = [-1,0] \cup \{ z_k, \; k \in \N \}$ where $(z_k)$ is a decreasing positive sequence tending to $0$ and if $\lim_{k \to \infty} z_{k-1}/z_k = 1$, then $\sigma$ is $\nabla$-differentiable on $\Tk$. In particular, we have $\sigma^\nabla (0) = 1$.
\item Application: if $\T = [-1,0] \cup \{ 1/k^2, \; k \in \N^* \}$, then $\sigma$ is $\nabla$-differentiable on $\Tk$. In particular, we have $\sigma^\nabla (0) = 1$.
\item Similarly to $3.$, we can prove that if $\T = \{ 0 \} \cup \{ z^-_k, \; k \in \N \} \cup \{ z^+_k, \; k \in \N \}$ where $(z^-_k)$ (resp. $(z^+_k)$) is an increasing negative (resp. decreasing positive) sequence tending to $0$ and if $\lim_{k \to \infty} z^-_{k+1}/z^-_k = \lim_{k \to \infty} z^+_{k-1}/z^+_k = \ell$, then $\sigma$ is $\nabla$-differentiable on $\Tk$. In particular, we have $\sigma^\nabla (0) = \ell$. Note that, in such a case, we can only have $\ell = 1$ since $z^-_{k+1}/z^-_k < 1 < z^+_{k-1}/z^+_k$ for every $k \in \N^*$.
\item Application: if $\T = \{ 0 \} \cup \{ z^-_k, \; k \in \N \} \cup \{ z^+_k, \; k \in \N \}$ with $z^-_k = - 1/k$ and $z^+_k = 1 /k^2$ for every $k \in \N$, then $\sigma$ is $\nabla$-differentiable on $\Tk$. In particular, we have $\sigma^\nabla (0) = 1$.
\end{enumerate}
\end{example}

Let us give some examples of time scale with continuous but non-$\nabla$-differentiable $\sigma$.

\begin{example}\label{ex4}
\begin{enumerate}
\item If $\T = \{ 0 \} \cup \{ 1/k!, \; k\in \N \}$, then $\sigma$ is not $\nabla$-differentiable in $0$ since $k! / (k-1)! = k$ tends to $+\infty$.
\item If $\T = [-1,0] \cup \{ 1/2^k, \; k\in \N \}$, then $\sigma$ is not $\nabla$-differentiable in $0$ since $2^k / 2^{k-1} = 2$ does not tend to $1$.
\item If $\T = \{ 0 \} \cup \{ \pm 1/2^k, \; k\in \N \}$, then $\sigma$ is not $\nabla$-differentiable in $0$ since $2^k / 2^{k-1} = 2$, $2^k / 2^{k+1} = 1/2$ and $2 \neq 1/2$.
\end{enumerate}
\end{example}

Examples~\ref{ex2}, \ref{ex3} and \ref{ex4} allow to get a better understanding of the restrictions imposed on a time scale by the $\nabla$-differentiability of $\sigma$. Indeed, we conclude that such a time scale has to satisfy the following properties:
\begin{itemize}
\item Due to the continuity of $\sigma$, no point can be right-scattered and left-dense.
\item Due to the $\nabla$-differentiability of $\sigma$, the density in a dense point cannot be \textit{too weak}, in contrary to $1.$ in Example~\ref{ex4}. Secondly, in a left- and right-dense point, the left and the right densities have to be \textit{homogeneous} with limit equal to $1$, as in $8.$, $10.$ of Example~\ref{ex3} and in contrary to $2.$, $3.$ of Example~\ref{ex4}.
\end{itemize}

\subsection{Some results and proof of Theorem~\ref{thmmain}}\label{section23}

The most important result of this section is the following one.

\begin{proposition}\label{thm1}
Let $u : \T \longrightarrow \R^n$ and let $t \in \TKk$. If the two following properties are satisfied:
\begin{itemize}
\item $\sigma$ is $\nabla$-differentiable at $t$;
\item $u$ is $\DD$-differentiable at $t$;
\end{itemize}
then, $u^\sigma$ is $\nabla$-differentiable at $t$ with $(u^\sigma)^\nabla (t) = \sigma^\nabla (t) u^\DD (t)$.
\end{proposition}

\begin{proof}
Since $\sigma$ is continuous at $t$, recall that $\sigma \circ \rho (t) = t$ from Lemma~\ref{propcont}. We distinguish two cases: $t \in \LS$ and $\rho(t)=t$.
\begin{itemize}
\item Firstly, let us consider that $t \in \LS$. Since $\sigma$ is continuous at $t$, we have $\sigma^\nabla (t) = \mu (t) / \nu (t)$, see Proposition~\ref{propcont2}. If moreover $t \in \RS$, then $u^\DD (t) = (u(\sigma(t)) - u(t))/\mu (t)$ and since $t$ is isolated, $u^\sigma$ is $\nabla$-differentiable at $t$ with
\begin{equation}
(u^\sigma)^\nabla (t) = \dfrac{u^\sigma(t) - u^\sigma (\rho(t))}{\nu (t)} = \dfrac{u(\sigma(t)) - u(t)}{\nu (t)} = \dfrac{\mu (t)}{\nu (t)} u^\DD (t) = \sigma^\nabla (t) u^\DD (t).
\end{equation}
In the contrary case $\sigma(t) = t$, since $\sigma$ is continuous at $t$ and since $u$ is continuous at $t = \sigma (t)$, we deduce that $u^\sigma$ is continuous at $t \in \LS$. Then, from the nabla version of Proposition~\ref{proprappeldelta}, $u^\sigma$ is $\nabla$-differentiable at $t$ with
\begin{equation}
(u^\sigma)^\nabla (t) = \dfrac{u^\sigma(t) - u^\sigma (\rho(t))}{\nu (t)} = \dfrac{u(\sigma(t)) - u(t)}{\nu (t)} = 0,
\end{equation}
since $\sigma (t) = t$. However, in this case, we have $\sigma^\nabla (t) = \mu (t) / \nu (t) = 0$. Consequently, we also retrieve $(u^\sigma)^\nabla (t) = \sigma^\nabla (t) u^\DD (t)$ in this case.
\item Secondly, let us consider that $\rho(t)=t$. Necessarily, $t \in \RD$. Indeed, if $t=a \in \Tk$, then $a \in \RD$. If $t \neq a$, then $t \in \LD$ and since $\sigma$ is continuous at $t$ and $t \in \Tk$, we deduce that $t \in \RD$ from Lemma~\ref{propcont}. Finally, since $u$ is $\DD$-differentiable at $t$ and since $\sigma$ is $\nabla$-differentiable at $t$, we have
\begin{equation}
(u^\sigma)^\nabla (t) = \lim\limits_{\substack{s \to t \\ s \neq t}} \dfrac{u^\sigma(s) - u^\sigma (\rho(t))}{s-\rho(t)} =  \lim\limits_{\substack{s \to t \\ s \neq t}} \dfrac{\sigma(s) - t}{s-t} \; \dfrac{u(\sigma(s)) - u (t)}{\sigma(s)-t}  = \sigma^\nabla (t) u^\DD (t).
\end{equation}
In the previous limit, since $\sigma$ is continuous at $t \in \LD \cap \RD$, we have used that $s \to t$, $s \neq t$ implies that $\sigma (s) \to \sigma (t) = t$, $\sigma(s) \neq t$.
\end{itemize}
The proof is complete.
\end{proof}

From Proposition~\ref{thm1}, the following corollary is directly derived.

\begin{corollary}\label{cor1}
Let $u : \T \longrightarrow \R^n$. If the following properties are satisfied:
\begin{itemize}
\item $\sigma$ is $\nabla$-differentiable on $\Tk$;
\item $u$ is $\DD$-differentiable on $\TK$;
\end{itemize}
then, $u^\sigma$ is $\nabla$-differentiable at every $t \in \TKk$ with $(u^\sigma)^\nabla (t) = \sigma^\nabla (t) u^\DD (t)$.
\end{corollary}

From Propositions~\ref{proprappeldelta2}, \ref{propintel} and Corollary~\ref{cor1}, our main result (Theorem~\ref{thmmain}) is proved. We conclude this section by introducing the following Leibniz formula useful in order to derive a Noether-type theorem (Theorem~\ref{thmnoether}) in Section~\ref{section3}.

\begin{proposition}[Leibniz formula]\label{propleibniz}
Let $u, v : \T \longrightarrow \R^n$ and $t \in \TKk$. If the following properties are satisfied:
\begin{itemize}
\item $\sigma$ is $\nabla$-differentiable at $t$;
\item $u$ is $\DD$-differentiable at $t$;
\item $v$ is $\nabla$-differentiable at $t$;
\end{itemize}
then, $u^\sigma \cdot v$ is $\nabla$-differentiable at $t$ and the following Leibniz formula holds:
\begin{equation}
(u^\sigma \cdot v )^\nabla (t) = u(t) \cdot v^\nabla (t) + \sigma^\nabla (t) u^\DD (t) \cdot v(t).
\end{equation}
\end{proposition}

\begin{proof}
Since $\sigma$ is continuous at $t \in \Tk$, we have $\sigma \circ \rho (t) = t$ from Lemma~\ref{propcont}. From Proposition~\ref{thm1}, $u^\sigma$ is $\nabla$-differentiable at $t$ with $(u^\sigma)^\nabla (t) = \sigma^\nabla (t) u^\DD (t)$. Finally, from the usual Leibniz formula on time scale (see \cite[Theorem~1.20 p.8]{bohn}), we have $u^\sigma \cdot v$ is $\nabla$-differentiable at $t$ with
\begin{equation}
(u^\sigma \cdot v )^\nabla (t) = u^\sigma(\rho(t)) \cdot v^\nabla (t) + (u^\sigma )^\nabla (t) \cdot v(t) = u(t) \cdot v^\nabla (t) + \sigma^\nabla (t) u^\DD (t) \cdot v(t).
\end{equation} 
The proof is complete.
\end{proof}

\section{Application to a Noether-type theorem}\label{section3}

We first review the definition of a one-parameter family of infinitesimal transformations of $\R^n$.

\begin{definition}
Let $\eta > 0$. A map $\Phi$ is said to be a one-parameter family of infinitesimal transformations of $\R^n$ if $\Phi$ is a map of class $\CC^2$
\begin{equation}
\fonction{\Phi}{[-\eta,\eta] \times \R^n}{\R^n}{(\theta,x)}{\Phi (\theta,x),}
\end{equation}
such that $\Phi (0,\cdot) = \mathrm{Id}_{\R^n}$.
\end{definition}

The action of a one-parameter family of infinitesimal transformations of $\R^n$ on a Lagrangian allows us to introduce the notion of symmetry for a $\nabla \circ \DD$-differential Euler-Lagrange equation~\eqref{eqdiffel}.

\begin{definition}
Let $\Phi$ be a one-parameter family of infinitesimal transformations of $\R^n$. A Lagrangian $L$ is said to be invariant under the action of $\Phi$ if for every solution $u$ of \eqref{eqdiffel} and every $t \in \TKk$, the map
\begin{equation}\label{eqinvar}
\theta \longmapsto L(\Phi (\theta,u(t)),\Phi (\theta,u)^\DD (t),t)
\end{equation}
has a null derivative at $\theta = 0$. In such a case, $\Phi$ is said to be a symmetry of the $\nabla \circ \DD$-differential Euler-Lagrange equation \eqref{eqdiffel} associated.
\end{definition}

The most classical examples of invariance of a Lagrangian under the action of a one-parameter family of infinitesimal transformations of $\R^n$ are given by quadratic Lagrangians and rotations:

\begin{example}\label{ex5}
Let us consider $n=2$, $L(x,v,t) = \Vert x \Vert^2 + \Vert v \Vert^2$, $\eta = \pi > 0$ and $\Phi$ defined by
\begin{equation}
\fonction{\Phi}{[-\pi,\pi] \times \R^2}{\R^2}{(\theta,x_1,x_2)}{\left( \begin{array}{cc} \cos (\theta) & -\sin(\theta) \\ \sin(\theta) & \cos (\theta) \end{array} \right) \times \left( \begin{array}{c} x_1 \\ x_2 \end{array} \right).}
\end{equation}
Then, for every $u \in \CCC^{1,\DD}_\mathrm{rd} (\T)$, every  $(\theta,t) \in [-\pi,\pi] \times \TKk$, we have $\Phi (\theta,u)^\DD (t) = \Phi (\theta,u^\DD(t))$ from linearity and continuity of $\Phi$ in its two last variables. Consequently, for every $u \in \CCC^{1,\DD}_\mathrm{rd} (\T)$ and every $\TKk$, one can easily prove that the map
\begin{equation}
\theta \longmapsto L(\Phi (\theta,u(t)),\Phi (\theta,u)^\DD (t),t)
\end{equation}
is independent of $\theta$ and consequently is constant and has a null derivative at $\theta = 0$.
\end{example}

Finally, on time scales with $\nabla$-differentiable $\sigma$, we prove the following Noether-type theorem providing a constant of motion for $\nabla \circ \DD$-differential Euler-Lagrange equations \eqref{eqdiffel} admitting a symmetry.

\begin{theorem}[Noether]\label{thmnoether}
Let us assume that $\sigma$ is $\nabla$-differentiable on $\Tk$ and let $\Phi$ be a one-parameter family of infinitesimal transformations of $\R^n$. If $L $ is invariant under the action of $\Phi$, then for every solution $u$ of \eqref{eqdiffel}, there exists $c \in \R$ such that
\begin{equation}
\dfrac{\partial L}{\partial v} (u(t),u^\DD (t),t) \cdot \dfrac{\partial \Phi}{\partial \theta} (0,u^\sigma (t))  =c,
\end{equation}
for every $t \in \TK$.
\end{theorem}

\begin{proof}
Let $u$ be a solution of \eqref{eqdiffel}. Let us differentiate the map given by~\eqref{eqinvar} at $\theta = 0$ and let us invert the operators $\DD$ and $\partial / \partial \theta$ from the $\CC^2$-regularity of $\Phi$. We obtain for every $t \in \TKk$:
\begin{equation}
\dfrac{\partial L}{\partial x} (u(t),u^\DD (t),t) \cdot \dfrac{\partial \Phi}{\partial \theta} (0,u(t)) + \dfrac{\partial L}{\partial v} (u(t),u^\DD (t),t) \cdot \dfrac{\partial \Phi}{\partial \theta} (0,u)^\DD (t) = 0.
\end{equation}
Finally, multiplying this last equality by $\sigma^\nabla (t)$ and using that $u$ is solution of \eqref{eqdiffel} on $\TKk$, we obtain
\begin{equation}
\left[ \dfrac{\partial L}{\partial v} (u,u^\DD,\cdot) \right]^\nabla (t) \cdot \dfrac{\partial \Phi}{\partial \theta} (0,u(t)) + \sigma^\nabla (t) \dfrac{\partial L}{\partial v} (u(t),u^\DD (t),t) \cdot \dfrac{\partial \Phi}{\partial \theta} (0,u)^\DD (t) = 0,
\end{equation}
for every $t \in \TKk$. Finally, from the Leibniz formula obtained in Proposition~\ref{propleibniz}, it holds
\begin{equation}
\left[ \dfrac{\partial L}{\partial v} (u,u^\DD,\cdot) \cdot \dfrac{\partial \Phi}{\partial \theta} (0,u)^\sigma \right]^\nabla (t) = 0,
\end{equation}
for every $t \in \TKk$. From the nabla version of Proposition~\ref{proprappeldelta2}, the proof is complete.
\end{proof}

Note that this theorem encompasses both the well known Noether's theorems given in the continuous case \cite[p.88]{arno} and in the (nonshifted) discrete case \cite[Theorem~6.4]{lubi}. For an example of application of Theorem~\ref{thmnoether}, one can consider the framework given in Example~\ref{ex5}.

\section{The $\nabla$-analogous results}\label{section4}
We conclude this paper with the following remark. The whole study made in this paper can be analogously derived for nonshifted calculus of variations with Lagrangian functionals involving a $\nabla$-integral dependent on a $\nabla$-derivative. We refer to the duality principle introduced in \cite{capu} and to \cite{torr15} for an example of application in calculus of variations on general time scales.

Precisely, let us assume that $\rho$ is $\DD$-differentiable on $\TK$. Then, the following $\DD \circ \nabla$-differential Euler-Lagrange equation on $\TKk$:
\begin{equation}\label{eqdiffelnabla}\tag{EL${}^{\DD \circ \nabla}_{\text{diff}}$}
\left[ \dfrac{\partial L}{\partial v} (u,u^\nabla,\cdot) \right]^\DD (t) =  \rho^\DD (t) \dfrac{\partial L}{\partial x} (u(t),u^\nabla(t),t)
\end{equation}
characterizes the critical points of the following (nonshifted) Lagrangian functional:
\begin{equation}
\fonction{\L}{\mathrm{C}^{1,\nabla}_{\mathrm{ld}}(\T)}{\R}{u}{\di \int_a^b L(u(\tau),u^\nabla(\tau),\tau) \; \nabla \tau.}
\end{equation}
In particular, a necessary condition for local optimizers of $\L$ is to be a solution of \eqref{eqdiffelnabla}. 

Finally, let us assume that $L $ is moreover invariant under the action of a one-parameter family $\Phi$ of infinitesimal transformations of $\R^n$ in the sense that for every solution $u$ of \eqref{eqdiffelnabla} and every $t \in \TKk$, the map
\begin{equation}
\theta \longmapsto L(\Phi (\theta,u(t)),\Phi (\theta,u)^\nabla (t),t)
\end{equation}
has a null derivative at $\theta = 0$. Then, for every solution $u$ of \eqref{eqdiffelnabla}, there exists $c \in \R$ such that
\begin{equation}
\dfrac{\partial L}{\partial v} (u(t),u^\nabla (t),t) \cdot \dfrac{\partial \Phi}{\partial \theta} (0,u^\rho (t))  =c,
\end{equation}
for every $t \in \Tk$.

\bibliographystyle{plain}
\bibliography{bibliototal}

\end{document}